\documentclass{birkjour}
\usepackage{subfigure}
\usepackage{verbatim}
 \newtheorem{thm}{Theorem}[section]
 \newtheorem{corollary}[thm]{Corollary}
 \newtheorem{lemma}[thm]{Lemma}
 \newtheorem{Proposition}[thm]{Proposition}
 \theoremstyle{definition}
 
 \theoremstyle{remark}
 
 \newtheorem{example}{Example}
 \numberwithin{equation}{section}

 \newcommand{\R}{\mathbb{R}}

\renewcommand{\L}{\mathcal{L}}

\begin{document}

%-------------------------------------------------------------------------
% editorial commands: to be inserted by the editorial office
%
%---------------------------------------------------------------------------
%Insert here the title, affiliations and abstract:
%

\title[Isoperimetric Inequalities in Normed Planes]
 {Isoperimetric Inequalities in Normed Planes}

%----------Author 1
\author[R.Segadas]{Rafael S. dos Santos}

\address{%
Departamento de Matem\'{a}tica- PUC-Rio\br
Rio de Janeiro, RJ, Brasil}
\email{rsegadas@gmail.com}

%-----------Author 2
\author[M.Craizer]{Marcos Craizer}

\address{%
Departamento de Matem\'{a}tica- PUC-Rio\br
Rio de Janeiro, RJ, Brasil}
\email{craizer@puc-rio.br}

\thanks{The authors wants to thank CNPq and CAPES (financial code 001) for financial support during the preparation of this manuscript. This manuscript arises from the Ph.D. dissertation of the first author under the supervision of the second author (\cite{Segadas}).\newline E-mail of the corresponding author: craizer@puc-rio.br}
%----------classification, keywords, date



\date{August 28, 2020}
%----------additions
%%% ----------------------------------------------------------------------

\begin{abstract}
The classical isoperimetric inequality can be extended to a general normed plane (\cite{Busemann1}). In the Euclidean plane, the defect in the 
isoperimetric inequality can be calculated in terms of the signed areas of some singular sets. In this paper we consider normed planes with piecewise smooth unit balls
and the corresponding class of admissible curves. For such an admissible curve, the singular sets are defined as projections in the subspaces of symmetric and constant width admissible curves. In this context, we obtain some improved isoperimetric inequalities whose equality hold for symmetric or constant width curves.
\end{abstract}

%%% ----------------------------------------------------------------------
\maketitle
%%% ----------------------------------------------------------------------
%\tableofcontents

\section{Introduction}

The isoperimetric inequality in the plane is an old problem: It states that 
\begin{equation*}
\frac{L^2}{4\pi}\geq A(\gamma),
\end{equation*}
for any simple convex curve $\gamma$, with equality holding only for circles. 
The isoperimetric inequality has been extended to an arbitrary normed plane by Busemann. 
In this case, the curve of a fixed area that minimizes length is not the unit circle, but the dual unit circle, also called {\it isoperimetrix} (\cite{Busemann1},\cite{Guggenheimer},\cite{Martini},\cite{Strang},\cite{Thompson96}). By duality, for a fixed area, the unit circle is the curve that minimizes the dual length.

Normed planes with smooth strictly convex unit circles were considered in \cite{CTB1} and normed planes with polygonal unit circles were considered in \cite{CTB2}. In both cases, the notion of dual length of a convex curve can be extended to a (signed) dual length in a class $\mathcal{C}$ of curves, whose elements are called {\it admissible} curves. In the strictly convex case, the class $\mathcal{C}$ of admissible curves consists of the smooth Lagrangian curves of degree $1$, while for the polygonal case, the admissible curves are polygons with sides parallel to the unit circle. 
Consider the subspaces $\mathcal{C}_{cw}^{0}\subset\mathcal{C}$ of symmetric curves and $\mathcal{C}_{sym}^{0}\subset\mathcal{C}$ of constant width curves, both with zero dual length. Then the mixed area defines an inner product in $\mathcal{C}$ such that $\mathcal{C}_{cw}^{0}$ and $\mathcal{C}_{sym}^{0}$ become orthogonal. 
In this paper we generalize these constructions to normed planes whose unit ball is piecewise smooth. 

For a normed plane with piecewise unit circle, we define a class $\mathcal{C}$ of admissible curves and denote by 
$\mathcal{C}_{cw}^{0}\subset\mathcal{C}$ the class of constant width curves with zero dual length. Then the orthogonal
projection $\pi_{cw}^0(\gamma)$ of $\gamma\in\mathcal{C}$ on $\mathcal{C}_{cw}^0$ with respect to the signed inner product given by the mixed area is a well-known object, the {\it Wigner caustic}, also called {\it area evolute} or {\it mid-point parallel tangent line} of $\gamma$ (\cite{Craizer},\cite{Craizer-Martini},\cite{Giblin08}). 

The orthogonal complement of $\mathcal{C}_{cw}^{0}$ in $\mathcal{C}$ can be orthogonally decomposed in $\mathcal{C}_{sym}^{0}$, 
the space of symmetric curves with zero dual length and the one-dimensional subspace $\mathcal{U}$ consisting of multiples of the unit 
circle. In the euclidean case, the orthogonal projection $\pi_{sym}^{0}(\gamma)$ of $\gamma\in\mathcal{C}$ on 
$\mathcal{C}_{sym}^{0}$ was called 
the {\it constant with measure set} of $\gamma$ in \cite{Zwier3}.

Based on the above decompositions, we prove a formula for the isoperimetric defect related to the signed areas of $\pi_{cw}^0(\gamma)$
and $\pi_{sym}^0(\gamma)$. Our approach, besides providing another proof of the isoperimetric inequality of \cite{Busemann1}, also gives rise to some improved isometric inequalities whose equalities hold only for symmetric or constant width curves. These results generalize
the results of \cite{Zwier1} and \cite{Zwier3} in the euclidean plane. We also give a proof of Lhuilier's inequality, taking advantage of the fact
that our proof holds for smooth unit circles as well as for polygonal unit circles.

The paper is organized as follows: In section 2 we define the admissible curves associated with the unitary ball of a normed plane. In section 3 we discuss mixed areas.
In section 4 we decompose the space of admissible curves with zero dual length in two orthogonal subspaces and discuss the properties of the projections of a curve in these subspaces. In section 5 we prove the improved isoperimetric inequalities, which are the main results of the paper.

\section{Unit Ball and Admissible Curves}

\subsection{Unit ball and its dual}

Consider a normed plane with unit ball $U$. We shall assume that the boundary of $U$, $u=\partial U$, is piecewise smooth. 
More precisely, we shall assume that $u=\bigcup_{1\leq i\leq 2n} u_i$, where $u_i$ are smooth arcs with $u_{i+n}=-u_i$. We shall also
assume that each $u_i$ is either a smooth strictly convex arc or a straight segment.

The dual ball $U^*$ is defined as
\begin{equation*}
U^*=\{ f\in (\R^2)^*|\ ||f||\leq 1\},
\end{equation*}
where $||f|| =\sup\{ |f(x)|\ |\ x\in U\}$. 
Each functional $f\in (\R^2)^*$ can be represented by a vector $v\in\R^2$ by the relation
\begin{equation}\label{eq:Identification}
f(x)=\left[  x,  v  \right], \ x\in\R^2,
\end{equation}
where $[a,b]$ denotes the determinant of the $2\times 2$ matrix whose columns are $a,b\in\R^2$. We shall represent $U^*$ by $V$ under this identification.

Given $v\in\R^2$, let $L_v$ denote the support line of $U$ parallel to $v$, and assume $L_v$ touches $U$ in $u\in\partial U$. Then
$v\in V$ if and only if $[u,v]=1$. Thus if $u$ belongs to a smooth arc, the corresponding $v$ also belongs to a smooth arc. If $u$ belongs
to a straight segment, then $v$ is constant. Finally if $u$ is a vertex, then $v$ describes a segment. An example of a piecewise smooth unit ball 
$U$ and its dual $V$ can be seen in Figure \ref{Fig:MixedBall}. 

\begin{figure}[htb]
\centering
\subfigure{
\includegraphics[width=.3\textwidth]{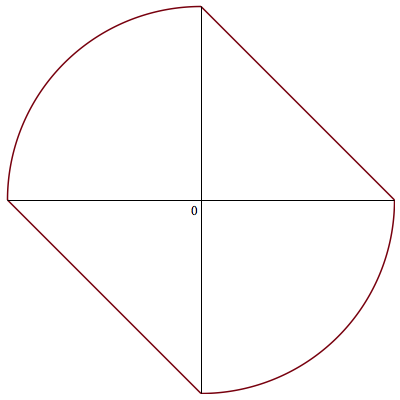}  }
% \subcaption{a. Support function of $\Gamma$.} }
\subfigure{
\includegraphics[width=.3\textwidth]{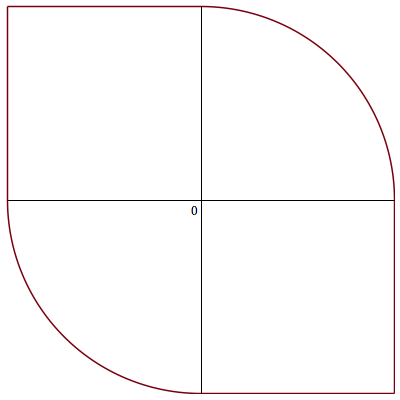} }
%  subcaption{b. Affine focal set.} }
% }
\caption{An example of a piecewise smooth unit ball and its dual.}
\label{Fig:MixedBall}
\end{figure}

\subsection{Parameterization of the unit circle}

Consider a parameterization $u(t)$ of $u$ such that $u'(t)\neq 0$, for any $t\in[t_i,t_{i+1}]$, where $u(t_i)$ are the vertices of $u$. There are $2$ types of intervals $[t_i,t_{i+1}]$: The smooth intervals,
where $u$ and $v$ are strictly convex, and the linear intervals, where $u$ is a straight segment and $v$ is constant. At smooth intervals we shall assume 
also that $[u'(t),u''(t)]\neq 0$. For any interval $t_i\leq t\leq t_{i+1}$, we can write
\begin{equation}\label{eq:DualCircle}
v(t)=\frac{u'(t)}{[u(t),u'(t)]}.
\end{equation}
In fact, the functional associated with $v(t)$ under the identification \eqref{eq:Identification} is zero in the direction $u'(t)$ and is one at $u(t)$. Thus it belongs to the dual unit circle at direction $u'(t)$. Since the straight lines are not included in Equation \eqref{eq:DualCircle}, this formula parameterizes only a part of the dual circle,  but
in fact this part is all that we need in the paper.
%For smooth or corner intervals, one can recover $u(t)$ from $v(t)$ by formula
%\begin{equation}\label{eq:RecoverU}
%u(t)=-\frac{v'(t)}{[v(t),v'(t)]}.
%\end{equation}

\begin{example}\label{ex:MixedBall}
Consider the unit ball shown in Figure  \ref{Fig:MixedBall}. Then a parameterization for the unit circle is given by 
\begin{equation*}\label{eq:MixedBall}
u(t)=
\left\{
\begin{array}{c}
(1-t, t), \ \ \ 0\leq t\leq 1,\\
%(0,1), \ \ \ 1\leq t\leq 2,\\
(\cos(\frac{\pi}{2}t), \sin(\frac{\pi}{2}t)),\ \ \   1\leq t\leq 2,\\
%(-1,0),\ \ \ 3\leq t\leq 4,\\
(t-3, 2-t), \ \ \ 2\leq t\leq 3,\\
%(0,-1), \ \ \ 5\leq t\leq 6,\\
(\cos(\frac{\pi}{2}t), \sin(\frac{\pi}{2}t)), \ \ \  3\leq t\leq 4, \\
%(1,0), \ \ \ 7\leq t\leq 8. 
\end{array}
\right.
\end{equation*}
The corresponding points in the dual circle are given by
\begin{equation*}\label{eq:MixedBallDual}
v(t)=
\left\{
\begin{array}{c}
(-1, 1), \ \ \ 0\leq t\leq 1,\\
%(-1,2-t), \ \ \ 1\leq t\leq 2,\\
(-\sin(\frac{\pi}{2}t), \cos(\frac{\pi}{2}t)),\ \ \   1\leq t\leq 2,\\
%(t-3,-1),\ \ \ 3\leq t\leq 4,\\
(1, -1), \ \ \ 2\leq t\leq 3,\\
%(1,t-6), \ \ \ 5\leq t\leq 6,\\
(-\sin(\frac{\pi}{2}t), \cos(\frac{\pi}{2}t)), \ \ \  3\leq t\leq 4.
%(7-t,1), \ \ \ 7\leq t\leq 8. 
\end{array}
\right.
\end{equation*}
Note that this parameterization does not include the straight lines of the dual ball.
\end{example}

\subsection{Admissible curves}

Denote by $\mathcal{C}=\mathcal{C}(U)$ the class of piecewise smooth closed curves $\gamma$ parameterized by $0\leq t\leq 2T$, $\gamma(0)=\gamma(2T)$ such that, for $t_i\leq t\leq t_{i+1}$,
\begin{equation}\label{eq:CurvatureRadius}
\gamma'(t)=r(t)u'(t),
\end{equation}
for some scalar function $r(t)$. We remark that when $u$ is smooth, the class $\mathcal{C}$ includes all the convex
smooth curves (\cite{CTB1}), and when $u$ is polygonal, the class $\mathcal{C}$ consists of all polygons with sides parallel to those of $u$ (\cite{CTB2}).

The scalar $r(t)=r_{\gamma}(t)$ is called the {\it curvature radius} of $\gamma$ at $\gamma(t)$ (\cite{Balestro},\cite{Petty1}). Up to sign, the curvature radius
is independent of the choice of the parametrizations of $\gamma$ and $u$. 
At an interval where $u$ is a straight segment, Equation \eqref{eq:CurvatureRadius} implies that $\gamma$ is also a straight segment and $r(t)$ is the rate between the lengths of the $\gamma$ and $u$ segments. At a smooth interval, we have that
$$
[\gamma'(t),\gamma''(t)]=r(t)^2[u',u''](t),
$$
which implies that $\gamma$ has no inflection points. 
It is clear from Equation \eqref{eq:CurvatureRadius} that, up to a translation, we can recover $\gamma$ from $r$. It is easy to see from the curvature radius whether or not the curve is convex, as next lemma shows:

\begin{lemma}\label{lemma:rpositive}
The curve $\gamma\in\mathcal{C}$ is convex if and only if $r(\gamma)$ does not change sign. 
\end{lemma}

\begin{proof}
Denote $\gamma'(t^+)=\lim_{s\to t^+}\gamma'(t)$ and $\gamma'(t^-)=\lim_{s\to t^-}\gamma'(s)$. If $r$ changes sign at a point $t_0$ of a smooth interval, then 
$\gamma'(t_0^+)$ and $\gamma'(t_0^-)$
are pointing in opposite directions and so the curve is not convex at this point. Similarly, if $r$ changes sign at a vertex $t_0$, then $\gamma'(t_0^+)$ and 
$\gamma'(t_0^-)$ make an angle bigger than $\pi$, and again the curve is not convex at this point.

Conversely, if $r$ does not change sign, the tangents at vertices are making angles strictly smaller than $\pi$ and so the curve is locally convex at 
the vertices. Since the curve has no inflection points, it is also locally convex at the smooth arcs. 
Finally, since the index of $\gamma$ is $\pm 1$, the curve $\gamma$ is necessarily convex.
\end{proof}

%Denote by 
%$\mathcal{C}_{conv}=\mathcal{C}_{conv}(U)$ the class of {\it convex} curves satisfying Equation \eqref{eq:CurvatureRadius}.
%By Lemma \ref{lemma:CurvatureConvex}, a curve $\gamma\in\mathcal{C}(U)$ belongs to $\mathcal{C}_{conv}(U)$ if and only if $r(t)\geq 0$, for any $t$.

\begin{corollary}\label{cor:ConvexParallel}
Given $\gamma\in\mathcal{C}$, there exists a constant $K>0$ such that $\gamma+Ku$ is convex.
\end{corollary}

\begin{proof}
The curvature radius of $\gamma+Ku$ is $r(t)+K$. If we choose $K\geq -\min\{r(t)\}$, we obtain a convex curve. 
\end{proof}

\subsection{Dual length}

Given $\gamma\in\mathcal{C}$,
from Equations \eqref{eq:DualCircle} and \eqref{eq:CurvatureRadius}, we can write
\begin{equation*}
\gamma'(t)=r(t)[u,u'](t)v(t),
\end{equation*}
for any $t$ in a smooth or linear interval. 
The (signed) {\it dual length} of $\gamma\in\mathcal{C}$ is defined by
\begin{equation}\label{eq:DualLength}
L_*(\gamma)=\int_0^{2T} r(t)[u,u'](t)dt.
\end{equation}
When $\gamma$ is convex, $L_*(\gamma)$ coincides with the dual length of $\gamma$.

\begin{example}\label{ex:AdmCurve}
Let  
\begin{equation*}\label{eq:AdmCurve}
\gamma(t)=
\left\{
\begin{array}{c}
(2-t,1+t), \ \ \ 0\leq t\leq 1,\\
%(1,2), \ \ \ 1\leq t\leq 2,\\
\frac{1}{\sqrt{15\cos^2(\frac{\pi}{2}t)+1}}(16\cos(\frac{\pi}{2}t),\sin(\frac{\pi}{2}t))+ (1,1),\ \ \   1\leq t\leq 2,\\
%(-3,1),\ \ \ 3\leq t\leq 4,\\
(-11+4t,9-4t), \ \ \ 2\leq t\leq 3,\\
%(1,-3), \ \ \ 5\leq t\leq 6,\\
\frac{1}{\sqrt{15\sin^2(\frac{\pi}{2}t)+1}}(\cos(\frac{\pi}{2}t),16\sin(\frac{\pi}{2}t))+ (1,1), \ \ \  3\leq t\leq 4, \\
%(2,1), \ \ \ 7\leq t\leq 8,
\end{array}
\right.
\end{equation*}
(see Figure \ref{fig:AdmCurve}).

\begin{figure}[htb]
\centering
\includegraphics[width=0.3\linewidth]{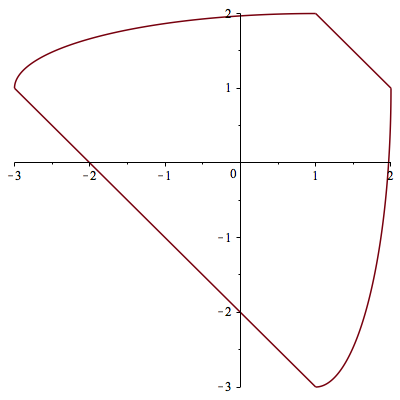}
 \caption{ A curve in $\mathcal{C}$ (Example \ref{ex:AdmCurve}). }
\label{fig:AdmCurve}
\end{figure}

The curvature radius is given by
\begin{equation*}\label{eq:SupAdmCurve}
r(t)=
\left\{
\begin{array}{c}
1, \ \ \ 0\leq t\leq 1,\\
%15t-14, \ \ \ 1\leq t\leq 2,\\
\frac{16}{\sqrt{(15\cos^2(\frac{\pi}{2}t)+1)^3}},\ \ \   1\leq t\leq 2,\\
%\frac{15}{4}t-11, \ \ \ 3\leq t\leq 4,\\
4, \ \ \ 2\leq t\leq 3,\\
%-\frac{15}{4}t+\frac{91}{4}, \ \ \ 5\leq t\leq 6,\\
\frac{16}{\sqrt{(15\sin^2(\frac{\pi}{2}t)+1)^3}}, \ \ \  3\leq t\leq 4,\\
%-15t+121, \ \ \ 7\leq t\leq 8.
\end{array}
\right.
\end{equation*}

We can calculate the dual length of $\gamma$ by equations \eqref{eq:DualLength} to obtain
\begin{equation*}
L_*(\gamma)\approx 13.58.
\end{equation*}
\end{example}

\section{Mixed areas}

\subsection{Basic properties}

Consider two curves $\gamma_1,\gamma_2\in\mathcal{C}$. Then the mixed area is defined by 
$$
A(\gamma_1,\gamma_2)=\frac{1}{2}\int_0^{2T} [\gamma_1,\gamma_2'](t) dt.
$$
The signed area of $\gamma\in\mathcal{C}$ is defined as $A(\gamma)=A(\gamma,\gamma)$. If $\gamma$ is convex, $A(\gamma)$ is the area of the region
bounded by $\gamma$. 

\begin{lemma}
$A(\gamma_1,\gamma_2)$ is a symmetric bilinear map in $\mathcal{C}$. Moreover,
\begin{equation}\label{eq:CompDualMixed}
L_*(\gamma)=2A(u,\gamma).
\end{equation}
\end{lemma}
\begin{proof}
Since $[\gamma_1,\gamma_2]'=[\gamma_1',\gamma_2]+[\gamma_1,\gamma_2']$ we obtain
$$
\int_0^{2T}[\gamma_1,\gamma_2'](t)dt=\int_0^{2T}[\gamma_2,\gamma_1'](t)dt.
$$
For the second assertion observe that
\begin{equation*}
L_*(\gamma)=\int_0^{2T} r(t)[u,u'](t)dt=\int_0^{2T} [u,\gamma'](t)dt=2A(u,\gamma),
\end{equation*}
thus proving the lemma.
\end{proof}

\subsection{Minkowski inequality}

The Minkowski inequality states that for {\it convex} curves $\gamma\in\mathcal{C}$, 
\begin{equation*}
A(\gamma_1,\gamma_2)^2\geq A(\gamma_1)A(\gamma_2),
\end{equation*}
with equality if and only if $\gamma_1$ is a multiple of $\gamma_2$. 

\begin{lemma}\label{lemma:DesigMink}
For any $\gamma\in\mathcal{C}$,
\begin{equation*}
L_*(\gamma)^2\geq 4A(\gamma)A(u),
\end{equation*}
with equality if and only if $\gamma$ is a multiple $u$. 
\end{lemma}

\begin{proof}
Since $\gamma\in\mathcal{C}$, by Corollary \ref{cor:ConvexParallel} there exists a constant $K>0$ such that $\gamma+Ku$ are convex. Then, by Minkowski inequality,
$$
A(\gamma+Ku,u)^2\geq A(\gamma+Ku)A(u).
$$
Since $A(\gamma+Ku)=A(\gamma)+K^2A(u)+2KA(\gamma,u)$ and 
$A(\gamma+Ku,u)=A(\gamma,u)+KA(u),$
we conclude that 
$$
A(\gamma,u)^2\geq A(\gamma)A(u),
$$
thus proving the lemma.
\end{proof}

\subsection{Curves of constant width}

We say that $\gamma\in\mathcal{C}$ is of {\it constant width} if 
\begin{equation*}
[\gamma,v](t+T)+[\gamma,v](t)=c,
\end{equation*} 
for some constant $c$. 

\begin{lemma}
For constant width curves, we have that 
$$
L_*(\gamma)=2cA(U)
$$
and so necessarily $c=w_{\gamma}$.
\end{lemma}
\begin{proof}
Since
$$
2A(\gamma,u)=\int_0^{2T} [\gamma,u']dt=\int_0^T [u,u']\left( [\gamma,v](t)+[\gamma(t+T),v(t+T)]\right)dt
$$
we conclude that $L_*(\gamma)=2cA(U)$. 
\end{proof}

\begin{corollary}
A curve $\gamma\in\mathcal{C}$ with zero dual length is of constant width if and only if
\begin{equation*}\label{eq:ConstantWidth}
[\gamma,v](t+T)+[\gamma,v](t)=0.
\end{equation*} 
\end{corollary}

\section{The subspaces of constant width and symmetric curves}

In this section we consider the mixed area as a signed inner product in $\mathcal{C}$.

\subsection{An orthogonal decomposition of $\mathcal{C}$}

Denote by $\mathcal{U}$ the $1$-dimensional subspace of $\mathcal{C}$ consisting of the constant multiples of the unit ball,
by $\mathcal{C}_{sym}^0\subset\mathcal{C}$ the subspace consisting of symmetric curves with respect to the origin with zero dual length and by 
$\mathcal{C}_{cw}^0\subset\mathcal{C}$ the subspace consisting of constant width curves, also with zero dual length. Then Equation \eqref{eq:CompDualMixed} 
implies that both $\mathcal{C}_{sym}^0$ and $\mathcal{C}_{cw}^0$ are
orthogonal to $\mathcal{U}$. 

\begin{lemma}
The subspaces $\mathcal{C}_{sym}^0$ and $\mathcal{C}_{cw}^0$ are orthogonal. 
\end{lemma}
\begin{proof}
If $\gamma_1\in\mathcal{C}_{cw}^0$ and $\gamma_2\in\mathcal{C}_{sym}^0$,
we have that 
$$
[\gamma_1(t+T),v(t+t)]+[\gamma_1(t),v(t)]=0, \ \ \ \gamma_2(t+T)=-\gamma_2(t).
$$
Thus 
$$
A(\gamma_1,\gamma_2)=\int_0^{T} [\gamma_1,\gamma_2'](t)dt+\int_0^{T} [\gamma_1,\gamma_2'](t+T)dt,
$$
$$
=\int_0^{T} r_2(t)[u,u'](t) \left( [\gamma_1(t),v(t)]-[\gamma_1(t+T),v(t)] \right) dt,
$$
$$
=\int_0^{T} r_2(t)[u,u'](t) \left( [\gamma_1(t),v(t)]+[\gamma_1(t+T),v(t+T)] \right) dt,
$$
which proves that $A(\gamma_1,\gamma_2)=0$.
\end{proof}

\subsection{Orthogonal projections}

For $\gamma\in\mathcal{C}$, let
$$
\gamma_1(t)=\frac{1}{2}\left( \gamma(t)+\gamma(t+T) \right)
$$
be the Wigner caustic of $\gamma$. Observe that 
$$
[\gamma_1(t),v(t)]+[\gamma_1(t+T),v(t+T)]=[\gamma_1(t),v(t)]-[\gamma_1(t),v(t)]=0,
$$
which implies that
$\gamma_1$ has constant width $0$, i.e., $\gamma_1\in\mathcal{C}_{cw}^0$. Moreover
$$
\gamma(t)-\gamma_1(t)=\frac{1}{2}\left( \gamma(t)-\gamma(t+T) \right)
$$
is symmetric, which means $\left(\gamma-\gamma_1\right)\in\mathcal{C}_{sym}^0\oplus\mathcal{U}$. We conclude that
\begin{equation*}
\pi_{cw}^0(\gamma)=\frac{1}{2}\left( \gamma(t)+\gamma(t+T) \right).
\end{equation*}
In other words, the Wigner caustic of $\gamma$ coincides with the orthogonal projection of $\gamma$ on ${\mathcal C}_{cw}^0$.

Now let
\begin{equation*}
\gamma_2(t)=\frac{1}{2}\left( \gamma(t)-\gamma(t+T)-w_{\gamma}u(t) \right),
\end{equation*}
where
\begin{equation*}
w_{\gamma}=\frac{L_*(\gamma)}{A_U}
\end{equation*}
is the {\it mean width} of $\gamma$. In the euclidean case, $\gamma_2$ was called the constant width measure set of $\gamma$
in \cite{Zwier3}. 
Is is clear that $\gamma_2$ is symmetric and 
$$
L_*(\gamma_2)=L_*(\gamma)-\frac{1}{2}\int_0^{2T}w_{\gamma}[u,u'](t)dt=0,
$$
thus proving that $\gamma_2\in\mathcal{C}_{sym}^0$. Moreover 
$$
\gamma(t)-\gamma_2(t)=\gamma_1(t)+\frac{1}{2}w_{\gamma}u(t) 
$$
has constant width, which means $\left(\gamma-\gamma_2\right)\in\mathcal{C}_{cw}^0\oplus\mathcal{U}$. We conclude that
\begin{equation*}
\pi_{sym}^0(\gamma)=\frac{1}{2}\left( \gamma(t)-\gamma(t+T)-w_{\gamma}u(t) \right).
\end{equation*}

For future reference, we remark that 
\begin{equation}\label{eq:Decomposition}
\gamma=\pi_{cw}^0(\gamma)+\pi_{sym}^0(\gamma)+\frac{w_{\gamma}}{2}u.
\end{equation}

\begin{Proposition}
We have that $A(\pi_{cw}^0(\gamma))\leq 0$ with equality if and only if $\gamma$ is symmetric and $A(\pi_{sym}^0(\gamma))\leq 0$
with equality if and only if $\gamma$ is of constant width.
\end{Proposition}

\begin{proof}
The Proposition follows immediately from the fact that 
$$
L_*(\pi_{cw}^0(\gamma))=L_*(\pi_{sym}^0(\gamma))=0
$$
and Lemma \ref{lemma:DesigMink}.
\end{proof}

\begin{example}
The $\pi_{cw}^0(\gamma)$ and $\pi_{sym}^0(\gamma)$ of the curve $\gamma$ of Example \ref{ex:AdmCurve} can be seen in Figure \ref{Fig:WCCWMS}. 

\begin{figure}[htb]
\centering
\subfigure{
\includegraphics[width=.3\textwidth]{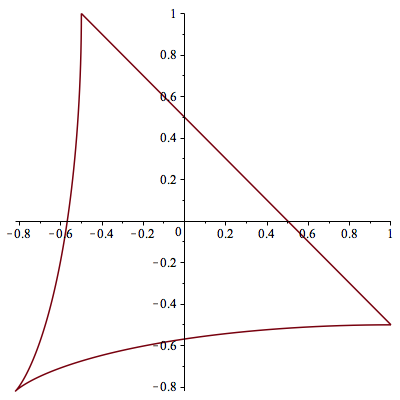}  }
% \subcaption{a. }}
\subfigure{
\includegraphics[width=.3\textwidth]{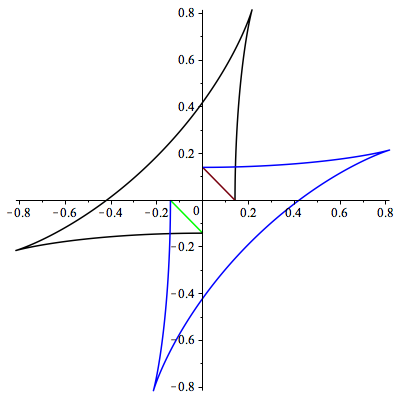} }
%  subcaption{b. Affine focal set.} }
% }
\caption{The projections $\pi_{cw}^0(\gamma)$ and $\pi_{sym}^0(\gamma)$ of the curve $\gamma$ of Example \ref{ex:AdmCurve}. Their areas 
are approximately $-1.33$ and $-0.48$, respectively. }
\label{Fig:WCCWMS}
\end{figure}

\end{example}

\section{Improved Isoperimetric Inequalities}

\subsection{An isoperimetric equality}

\begin{Proposition}\label{prop:Main}
Let $\gamma\in\mathcal{C}$ be a convex curve. Then
\begin{equation*}\label{eq:Isoperimetric}
\frac{L_*^2(\gamma)}{4A(U)}=A_{\gamma}-2A({\pi_{cw}^0(\gamma)})-A({\pi_{sym}^0(\gamma)}).
\end{equation*}
\end{Proposition}

\begin{proof}
Consider the orthogonal decomposition of $\gamma$ given by Equation \eqref{eq:Decomposition}. 
Taking into account that $WC$ is $T$-periodic, we obtain 
$$
A(\gamma)=2A(\pi_{cw}^0(\gamma))+A(\pi_{sym}^0(\gamma))+\frac{w^2_{\gamma}}{4}A(U).
$$
We conclude that
$$
A(\gamma)-2A(\pi_{cw}^0(\gamma))-A(\pi_{sym}^0(\gamma))=\frac{L^2_*(\gamma)}{4A(U)},
$$
thus proving the proposition.
\end{proof}

\subsection{Some consequences}

The next corollary gives us two new improved isoperimetric inequalities:

\begin{corollary}%\label{thm:Main}
Let $\gamma\in\mathcal{C}$ be a convex curve.
\begin{enumerate}
\item
The following improved isoperimetric inequality holds:
\begin{equation*}\label{eq:Isoperimetric}
\frac{L_*^2(\gamma)}{4A(U)}\geq A({\gamma})-A({\pi_{sym}^0(\gamma)}),
\end{equation*}
with equality if and only if $\gamma$ is symmetric.
\item
The following improved isoperimetric inequality holds:
\begin{equation*}\label{eq:Isoperimetric}
\frac{L_*^2(\gamma)}{4A(U)}\geq A({\gamma})-2A({\pi_{cw}^0(\gamma)}),
\end{equation*}
with equality if and only if $\gamma$ is of constant width.
\end{enumerate}
\end{corollary}

Next corollary recovers the isoperimetric inequality of Busemann in a more restricted case. 
In fact, Busemann's inequality holds for any convex curve $\gamma$ and any normed plane (see \cite{Busemann1}).

\begin{corollary}\label{cor:IsoIneq}
Let $\gamma\in\mathcal{C}$ be a convex curve. Then
\begin{equation}\label{eq:IsoIneq}
\frac{L_*^2(\gamma)}{4A(U)}\geq A({\gamma}),
\end{equation}
and equality holds if and only if $\gamma$ is a multiple of the unit ball.
\end{corollary}

\subsection{ Lhuilier's inequality}

In this section we prove Lhuilier's inequality, taking advantage of the fact that our result holds not only for smooth unit circles
but also for polygonal unit circles.

Consider a convex poligon $K$ and let $K_1$ be the polygon which is circumscribed about the unit circle and whose sides are respectively parallel to the sides of $K$. Let $L_*(K)$ denotes the dual length of $K$, 
$A(K)$ the area enclosed by $K$ and $A(K_1)$ the area enclosed by $K_1$. Then Lhuilier's theorem (\cite{Chakerian}) states that 
\begin{equation*}
\frac{L_*^2(K)}{4A(K_1)}\geq A(K),
\end{equation*}
with equality if and only if $K$ is a constant multiple of $K_1$.

We shall now proof a weaker version of Lhuilier's inequality, namely
\begin{equation}\label{eq:Lhuilier2}
\frac{L_*^2(K)}{4A(K_1^0)}\geq A(K),
\end{equation}
where $K_1^0=K_1\cap(-K_1)$ is the symmetrization of $K_1$, with equality if and only if $K$ is a constant multiple of $K_1^0$.
This inequality coincides with Lhuilier's inequality if $K$ has parallel opposite sides. To prove the general case of Lhuilier's inequality, 
we need to develop the results of this paper in the more general context of a non-symmetric unit ball. 

To prove Inequality \eqref{eq:Lhuilier2}, consider the normed plane with unit ball $K_1^0$. Observe that the dual length $L_*(K)$ with respect to original normed plane is the same as the dual length of $K$ with respect to the normed plane with unit ball $K_1^0$. Then by Inequality \eqref{eq:IsoIneq} applied
to $K\in\mathcal{C}(K_1^0)$, i.e., in the normed plane whose unit ball is $K_1^0$, we obtain Inequality \eqref{eq:Lhuilier2}.

% ------------------------------------------------------------------------
\end{document}